\documentclass{article}

\usepackage{amssymb}
\usepackage{amsmath}
\usepackage{amsthm}
\usepackage{a4wide}

\newtheorem{theorem}{Theorem}
\newtheorem{corollary}{Corollary}

\newtheorem{ex}{Example}
\theoremstyle{remark}
\newtheorem{remark}{Remark}

\theoremstyle{definition}
\newtheorem{definition}{Definition}

\begin{document}

\title{The Diamond Integral on Time Scales\thanks{This is a preprint of a paper 
whose final and definite form will appear in the 
\emph{Bulletin of the Malaysian Mathematical Sciences Society}. 
Paper submitted 12-Feb-2013; revised 07-May-2013; accepted for publication 05-Jun-2013.}}


\author{Artur M. C. Brito da Cruz$^{1, 2}$\\ {\sf artur.cruz@estsetubal.ips.pt}
\and
Nat\'{a}lia Martins$^{2}$\\ {\sf natalia@ua.pt}
\and
Delfim F. M. Torres$^{2}$\\ {\sf delfim@ua.pt}}

\date{$^{1}$Escola Superior de Tecnologia de Set\'{u}bal,
Estefanilha, 2910-761 Set\'{u}bal, Portugal\\[0.3cm]
$^{2}$CIDMA--Center for Research and Development in Mathematics and Applications,\\
Department of Mathematics, University of Aveiro, 3810-193 Aveiro, Portugal}

\maketitle


\begin{abstract}
We define a more general type of integral on time scales.
The new diamond integral is a refined version of the diamond-alpha
integral introduced in 2006 by Sheng et al. A mean value theorem for
the diamond integral is proved, as well as versions of Holder's,
Cauchy-Schwarz's and Minkowski's inequalities.

\bigskip


\smallskip

\end{abstract}


\section{Introduction}

In 1988, Hilger introduced the theory of time scales in his dissertation.
This intends to unify and extend existing continuous
and discrete calculus into a uniformed theory \cite{Hilger,Hilger:90}.
Since then, the time scale theory advanced fast, as seen by the number
of published works dedicated to the subject. Science Watch, from Thomson Reuters,
considered in October 2007, and later in February 2011,
the study of equations on time scales as an emerging research front in the field of mathematics
with ``potential applications in such areas as engineering, biology, economics and finance,
physics, chemistry, social sciences, medical sciences, mathematics education, and others''.
For a good introduction to the theory of time scales, we refer
the reader to the comprehensive books \cite{Bohner_1,Bohner_2}.

The delta/forward calculus and the nabla/backward calculus
were the first approaches of the calculus on time scales. In 2006, a combined diamond-$\alpha$
dynamic derivative (resp. integral) was introduced by Sheng, Fadag, Henderson and Davis,
as a linear combination of the delta and nabla dynamic derivatives (resp. integral) on time
scales \cite{Sheng}. Sheng et al. showed that those diamond-$\alpha$ derivatives offer
more balanced approximations for computational applications.
Having in mind that the classical symmetric derivative is a generalization
of the classical derivative and that the classical symmetric difference quotient
is generally a more accurate approximation for the derivative than the usual one-sided quotient,
an important question consists to define a symmetric derivative on time scales,
as a generalization of the diamond-$\alpha$ derivative \cite{Brito_da_Cruz_3}.
On the other hand, the problem of determining a symmetric integral
is a topic that captured the interest of many important mathematicians along the history,
such as Denjoy, James, Kurzweil and Jarn\'{\i}k.
However, it is generally accepted that the integrals proposed only
``invert approximately'' the corresponding derivative \cite{Thomson}.
Despite this apparent limitation, these integrals have nice properties.
For example, they are useful to solve the coefficient problem for trigonometric series
when the conventional integral fails to exist \cite{Thomson}.
Our goal here is to define the diamond integral as an
``approximate'' symmetric integral on time scales (see Section~\ref{int}).

The text is organized as follows. In Section~\ref{pre} we present some preliminary results
and basic definitions necessary in the sequel. Namely, we briefly present the nabla and
the delta calculus \cite{Bohner_1,Bohner_2}. We also recall the notions of
diamond-$\alpha$ derivative and integral, as a linear combination of the delta
and nabla derivatives and integral, respectively \cite{Sheng,Sheng_2,Sheng_3}.
In Section~\ref{int} we introduce our diamond integral,
derive some of its properties, and prove some integral inequalities.
We end with Section~\ref{sec:coc} of conclusion.


\section{Preliminaries}
\label{pre}

A nonempty closed subset of $\mathbb{R}$ is called
a time scale and is denoted by $\mathbb{T}$. We assume that a time scale
has the topology inherited from $\mathbb{R}$ with the standard topology.
Two jump operators are considered: the forward jump operator
$\sigma : \mathbb{T} \rightarrow \mathbb{T}$, defined by
$\sigma \left( t\right) :=\inf \left \{ s\in \mathbb{T}:s>t\right \}$
with $\inf \emptyset =\sup \mathbb{T}$ (i.e., $\sigma \left( M\right) =M$ if
$\mathbb{T}$ has a maximum $M$), and the backward jump operator
$\rho :\mathbb{T} \rightarrow \mathbb{T}$ defined by
$\rho \left( t\right) :=\sup \left \{ s\in \mathbb{T}:s<t\right \}$
with $\sup \emptyset =\inf \mathbb{T}$ (i.e.,
$\rho \left( m\right) =m$ if $\mathbb{T}$ has a minimum $m$).
If $\sup\mathbb{T}$ is finite and left-scattered, then we define
$\mathbb{T}^\kappa := \mathbb{T}\setminus \{\sup\mathbb{T}\}$,
otherwise $\mathbb{T}^\kappa :=\mathbb{T}$;
if $\inf\mathbb{T}$ is finite and right-scattered, then
$\mathbb{T}_\kappa := \mathbb{T}\setminus\{\inf\mathbb{T}\}$,
otherwise $\mathbb{T}_\kappa := \mathbb{T}$. We set
$\mathbb{T}_{\kappa}^{\kappa}:=\mathbb{T}_{\kappa}\cap \mathbb{T}^{\kappa}$.

\begin{definition}[\protect \cite{Bohner_1}]
We say that a function $f:\mathbb{T} \rightarrow \mathbb{R}$
is delta differentiable at $t\in \mathbb{T}^{\kappa}$ if there exists a
number $f^{\Delta }\left( t\right) $ such that, for all $\varepsilon >0$,
there exists a neighborhood $U$ of $t$ such that
\begin{equation*}
\left \vert f\left( \sigma \left( t\right) \right) -f\left( s\right)
-f^{\Delta }\left( t\right) \left( \sigma \left( t\right) -s\right) \right
\vert \leqslant \varepsilon \left \vert \sigma \left( t\right) -s\right\vert
\end{equation*}
for all $s \in U$. We call $f^{\Delta }\left( t\right)$
the delta derivative
of $f$ at $t$ and we say that $f$ is delta differentiable
if $f$ is delta differentiable for all $t\in \mathbb{T}^{\kappa}$.
\end{definition}

\begin{definition}[\protect \cite{Bohner_1}]
We say that a function $f:\mathbb{T} \rightarrow \mathbb{R}$
is nabla differentiable at $t\in \mathbb{T}_{\kappa}$ if there exists a
number $f^{\nabla }\left( t\right) $ such that, for all $\varepsilon >0$,
there exists a neighborhood $V$ of $t$ such that
\begin{equation*}
\left \vert f\left( \rho \left( t\right) \right) -f\left( s\right)
-f^{\nabla }\left( t\right) \left( \rho \left( t\right) -s\right) \right
\vert \leqslant \varepsilon \left \vert \rho \left( t\right) -s\right \vert
\end{equation*}
for all $s\in V$. We call $f^{\nabla }\left( t\right)$
the nabla derivative
of $f$ at $t$ and we say that $f$ is nabla differentiable
if $f$ is nabla differentiable for all $t\in \mathbb{T}_{\kappa}$.
\end{definition}

\begin{remark}
If $\mathbb{T} = \mathbb{R}$, then $f^\Delta= f^\nabla=f^\prime$,
where $f^\prime$ denotes the usual derivative on $\mathbb{R}$.
If $\mathbb{T} = \mathbb{Z}$, then $f^\Delta (t)= f(t+1)-f(t)$
and $f^\nabla (t)= f(t)-f(t-1)$, i.e.,  $f^\Delta$ and $f^\nabla$ are,
respectively, the usual forward and backward difference operators.
For any time scale $\mathbb{T}$, if $f$ is a constant function,
then $f^\Delta = f^\nabla \equiv 0$; if $f(t) = kt$ for some constant $k$,
then $f^\Delta = f^\nabla \equiv k$.
\end{remark}

Let $a,b \in \mathbb{T}$, $a<b$. In what follows we denote
$\left[a,b\right]_{\mathbb{T}}:=\{t \in \mathbb{T}: a \leq t \leq b\}$.

\begin{definition}[\protect \cite{Bohner_1}]
A function $F:\mathbb{T}\rightarrow \mathbb{R}$ is said to be a
delta antiderivative of $f:\mathbb{T}\rightarrow \mathbb{R}$, provided
$F^{\Delta }\left( t\right) =f\left( t\right)$
for all $t\in \mathbb{T}^{\kappa}$. For all $a,b\in\mathbb{T}$, $a<b$,
the delta integral of $f$ from $a$ to $b$
(or on $\left[a,b\right]_{\mathbb{T}}$) is defined by
\begin{equation*}
\int_{a}^{b}f\left( t\right) \Delta t=F\left( b\right) -F\left( a\right).
\end{equation*}
\end{definition}

\begin{definition}[\protect \cite{Bohner_1}]
A function $G:\mathbb{T} \rightarrow \mathbb{R}$
is said to be a nabla antiderivative
of $g:\mathbb{T} \rightarrow \mathbb{R}$, provided
$G^{\nabla }\left( t\right) =g\left( t\right)$
for all $t\in \mathbb{T}_{{k}}$. For all $a,b\in\mathbb{T}$, $a<b$,
the nabla integral of $g$ from $a$ to $b$ (or on $\left[a,b\right]_{\mathbb{T}}$)
is defined by
\begin{equation*}
\int_{a}^{b}g\left( t\right) \nabla t=G\left( b\right) -G\left( a\right).
\end{equation*}
\end{definition}

For the properties of the delta and nabla integrals
we refer the readers to \cite{Bohner_1,Bohner_2}.

\begin{remark}
If $\mathbb{T}=\mathbb{R}$, then
$\displaystyle \int_{a}^{b}f\left( t\right) \Delta t
= \int_{a}^{b}f\left( t\right) \nabla t
=\int_{a}^{b}f\left( t\right) dt$,
where the last integral is the usual Riemman integral.
If $\mathbb{T}= h\mathbb{Z}$, for some $h>0$,
and $a,b \in \mathbb{T}$, $a<b$, then
$\displaystyle \int_{a}^{b}f\left( t\right) \Delta t
=\sum_{k=\frac{a}{h}}^{\frac{b}{h}-1} hf\left( k h\right)$
and $\displaystyle \int_{a}^{b}f\left( t\right) \nabla t
=\sum_{k=\frac{a}{h}+1}^{\frac{b}{h}}hf\left( k h\right)$.
\end{remark}

To provide a shorthand notation, for a function $f:\mathbb{T}\rightarrow\mathbb{R}$
we let $f^{\sigma}(t):=f(\sigma(t))$ and $f^{\rho}(t):=f(\rho(t))$.

\begin{definition}[\protect \cite{Sheng_2}]
Let $t,s\in \mathbb{T}$ and define $\mu_{ts}:=\sigma \left( t\right) -s$
and $\eta_{ts}:=\rho \left( t\right) -s$.
We say that a function $f:\mathbb{T}\rightarrow\mathbb{R}$
is diamond-$\alpha$ differentiable at
$t\in\mathbb{T}_{\kappa}^{\kappa}$ if there exists a number
$f^{\diamondsuit _{\alpha }}\left(t\right)$
such that, for all $\varepsilon >0$, there exists
a neighborhood $U$ of $t$ such that, for all $s\in U$,
\begin{equation*}
\left \vert \alpha \left[ f^{\sigma }\left( t\right) -f\left( s\right) \right]
\eta _{ts}+\left( 1-\alpha \right) \left[ f^{\rho }\left( t\right) -f\left(
s\right) \right] \mu _{ts}-f^{\diamondsuit_{\alpha }}\left(t \right)
\mu_{ts}\eta_{ts}\right \vert \leqslant \varepsilon \left \vert \mu _{ts}\eta_{ts}\right \vert .
\end{equation*}
A function $f$ is said to be diamond-$\alpha$ differentiable provided
$f^{\diamondsuit _{\alpha }}\left(t\right)$ exists for all
$t\in \mathbb{T}_{\kappa}^{\kappa}$.
\end{definition}

\begin{theorem}[\protect \cite{Sheng_2}]
Let $0\leqslant \alpha \leqslant 1$ and let $f$ be both nabla and delta
differentiable at $t\in \mathbb{T}_{\kappa}^{\kappa}$. Then $f$ is
diamond-$\alpha$ differentiable at $t$ and
\begin{equation}
\label{ts:2}
f^{\diamondsuit_{\alpha }}\left( t\right) = \alpha f^{\Delta }\left(
t\right) +\left( 1-\alpha \right) f^{\nabla }\left( t\right).
\end{equation}
\end{theorem}

\begin{remark}
If $\alpha =1$, then the diamond-$\alpha$ derivative reduces to
the delta derivative; if $\alpha=0$, then the diamond-$\alpha$
derivative coincides with the nabla derivative.
\end{remark}

\begin{remark}
The equality \eqref{ts:2} is given as the definition
of the diamond-$\alpha$ derivative in \cite{Sheng}.
\end{remark}

\begin{definition}[\protect \cite{Sheng_2}]
Let $a,b\in \mathbb{T}$, $a<b$, $h:\mathbb{T} \rightarrow \mathbb{R}$
and $\alpha \in \left[ 0,1 \right]$. The diamond-$\alpha$ integral
(or $\diamondsuit _{\alpha}$-integral) of $h$ from $a$ to $b$
(or on $\left[a,b\right]_\mathbb{T}$) is defined by
\begin{equation*}
\int_{a}^{b}h\left( t\right) \diamondsuit _{\alpha }t=\alpha
\int_{a}^{b}h\left( t\right) \Delta t+\left( 1-\alpha \right)
\int_{a}^{b}h\left( t\right) \nabla t,
\end{equation*}
provided $h$ is delta and nabla integrable on $\left[a,b\right]_\mathbb{T}$.
\end{definition}

For properties, results, and integral inequalities
concerning the diamond-$\alpha$ integral, we refer the
reader to \cite{Malinowska,Rui,Moz:Tor:09,MR2609350}
and references therein.


\section{The diamond integral}
\label{int}

In the classical calculus, one can find several attempts to define a symmetric integral
--- see, e.g., \cite{Thomson,Thomson_2}. However, those integrals invert
only ``approximately'' the symmetric derivatives \cite{Thomson_2}.
In the quantum setting, symmetric integrals are available,
namely the $q$-symmetric integral --- see, e.g., \cite{Kac,Brito_da_Cruz_2}
---  and the Hahn symmetric integral, that inverts
the Hahn symmetric derivative \cite{Brito_da_Cruz_4}.
In the more general setting of time scales, the problem of
determining a symmetric integral is an interesting open question.
In particular, such integral would be new, and of great interest,
even in the classical case $\mathbb{T} = \mathbb{R}$.
Given the similarities and the advantages of the recent symmetric
derivative with respect to the diamond-$\alpha$ derivative \cite{Brito_da_Cruz_3},
we claim that the diamond integral here introduced bring us closer
to the construction of a genuine symmetric integral on time scales.

We borrow from \cite{Brito_da_Cruz_3} the real function
\begin{equation*}
\gamma\left( t\right) := \lim_{s\rightarrow t}\frac{\sigma \left(
t\right) -s}{\sigma \left( t\right) +2t-2s-\rho \left( t\right)},
\end{equation*}
which plays an important role in the definition of our diamond integral (Definition~\ref{m:d}).
Note that $\gamma$ is well defined, $0\leqslant\gamma\left( t\right)\leqslant 1$
for all $t\in \mathbb{T}$, and
$$
\gamma\left( t\right) =
\begin{cases}
\frac{1}{2} & \text{ if } t \text{ is dense}, \\
\frac{\sigma \left( t\right) -t}{\sigma \left( t\right)
-\rho \left(t\right)} & \text{ if } t \text{ is not dense}.
\end{cases}
$$

\begin{definition}[diamond integral]
\label{m:d}
Let $f:\mathbb{T\rightarrow \mathbb{R}}$ and $a,b\in \mathbb{T}$, $a<b$.
The diamond integral (or $\diamondsuit$-integral) of
$f$ from $a$ to $b$ (or on $[a,b]_{\mathbb{T}}$) is given by
\begin{equation*}
\int_{a}^{b}f\left( t\right) \diamondsuit t
:= \int_{a}^{b}\gamma\left(t\right) f\left( t\right) \Delta t
+\int_{a}^{b}\left(1-\gamma \left( t\right)\right) f\left( t\right) \nabla t,
\end{equation*}
provided $\gamma f$ is delta integrable and $(1-\gamma) f$ is nabla
integrable on $[a,b]_{\mathbb{T}}$ .
We say that the function $f$ is diamond integrable (or $\diamondsuit$-
integrable) on $[a,b]_{\mathbb{T}}$ if it is $\diamondsuit$-integrable for
all $a,b\in \mathbb{T}$.
\end{definition}

\begin{remark}
The $\diamondsuit$-integral coincides with the
$\diamondsuit_{\alpha }$-integral when the function $\gamma$
is constant and equal to $\alpha$. There are several important time scales
where this happens. For instance, when $\mathbb{T} = \mathbb{R}$
or $\mathbb{T} = h \mathbb{Z}$, $h>0$, the $\diamondsuit$-integral
is equal to the $\diamondsuit_{\frac{1}{2}}$-integral.
Since the fundamental theorem of calculus is not valid for the
$\diamondsuit_{\alpha}$-integral (see \cite{Sheng}),
it is clear that the fundamental theorem of calculus
is also not valid for the $\diamondsuit$-integral.
Hence, the diamond integral is not a genuine symmetric integral on time scales.
Despite this limitation, we show that the new integral satisfies some important properties.
\end{remark}

\begin{ex}
Let $f:\mathbb{Z}\rightarrow\mathbb{R}$ be defined by
$f\left( t\right) =t^{2}$. Then,
$$
\int_{0}^{2}f\left( t\right) \diamondsuit t
=\frac{1}{2} \int_{0}^{2}f\left( t\right) \Delta t
+\frac{1}{2}\int_{0}^{2}f\left(t\right) \nabla t
= \frac{1}{2}\sum_{t=0}^{1}f\left( t\right) +\frac{1}{2}\sum_{t=1}^{2}f\left( t\right)
= \frac{1}{2}\left( 0+1\right) +\frac{1}{2}\left( 1+4\right)=3.
$$
\end{ex}

\begin{ex}
Let $f:\left[ 0,1\right] \cup \left \{ 2,4\right \} \rightarrow \mathbb{R}$
be defined by $f\left( t\right) =1$. Then,
\begin{equation*}
\begin{split}
\int_{0}^{4}f\left( t\right) \diamondsuit t
&= \int_{0}^{1}1dt
+\int_{1}^{2}\gamma\left( t\right) \Delta t+ \int_{2}^{4}\gamma\left( t\right) \Delta t
+ \int_{1}^{2}(1-\gamma(t)) \nabla t + \int_{2}^{4}(1-\gamma(t)) \nabla t\\
&= 1+ \gamma(1)+2\cdot \gamma(2)+ (1-\gamma(2))+ 2\cdot(1-\gamma(4))=\frac{17}{3}.
\end{split}
\end{equation*}
Note that, in this example, and independently of the value of $\alpha$,
our diamond integral is different from the diamond-$\alpha$ integral:
$$
\int_{0}^{4}f\left( t\right) \diamondsuit_{\alpha }t
=\int_{0}^{1}f\left(t\right) \diamondsuit t+\alpha \int_{1}^{4}1\Delta t
+\left( 1-\alpha \right) \int_{1}^{4}1\nabla t
=1+3\alpha +\left( 1-\alpha \right) 3 = 4 \ne \frac{17}{3}.
$$
\end{ex}

The $\diamondsuit$-integral has the following properties.

\begin{theorem}
\label{ts:propriedades:integral}
Let $f,g:\mathbb{T}\rightarrow \mathbb{R}$ be $\diamondsuit$-integrable
on $\left[ a,b\right]_{\mathbb{T}}$. Let $c\in \left[a,b\right]_{\mathbb{T}}$
and $\lambda \in \mathbb{R}$. Then,
\begin{enumerate}
\item $\displaystyle \int_{a}^{a}f\left( t\right) \diamondsuit t=0$;

\item $\displaystyle \int_{a}^{b}f\left( t\right) \diamondsuit
t=\int_{a}^{c}f\left( t\right) \diamondsuit t+\int_{c}^{b}f\left( t\right)
\diamondsuit t$;

\item $\displaystyle \int_{a}^{b}f\left( t\right) \diamondsuit t=-
\displaystyle \int_{b}^{a}f\left( t\right) \diamondsuit t$;

\item $f+g$ is $\diamondsuit$-integrable on $\left[ a,b\right]_{\mathbb{T}}$ and
$\displaystyle \int_{a}^{b}\left( f+g\right) \left( t\right) \diamondsuit
t=\int_{a}^{b}f\left( t\right) \diamondsuit t+\int_{a}^{b}g\left( t\right)
\diamondsuit t$;

\item $\lambda f$ is $\diamondsuit$-integrable on $\left[ a,b\right]_{\mathbb{T}} $ and
$\displaystyle \int_{a}^{b}\lambda f\left( t\right) \diamondsuit t=\lambda
\int_{a}^{b}f\left( t\right) \diamondsuit t$;

\item $fg$ is $\diamondsuit$-integrable on $\left[ a,b\right]_{\mathbb{T}} $;

\item for $p>0$, $|f|^{p}$ is $\diamondsuit$-integrable on $\left[ a,b\right]_{\mathbb{T}}$;

\item if $f\left(t\right) \leqslant g\left( t\right)$
for all $t\in \left[ a,b\right]_{\mathbb{T}}$, then
$\displaystyle \int_{a}^{b}f\left( t\right) \diamondsuit t\leqslant \int_{a}^{b}g\left(
t\right) \diamondsuit t$;

\item $\left \vert f\right \vert $ is $\diamondsuit$-integrable
on $\left[ a,b\right]_{\mathbb{T}}$ and
$\displaystyle \left \vert \int_{a}^{b}f\left( t\right) \diamondsuit t\right \vert
\leqslant \int_{a}^{b}\left \vert f\left( t\right) \right \vert \diamondsuit t$.
\end{enumerate}
\end{theorem}

\begin{proof}
The results follow straightforwardly from the analogous
properties of the nabla and delta integrals.
\end{proof}

Next we extend to the diamond integral
some results obtained in \cite{Malinowska,Rui}
for the $\diamondsuit_{\alpha}$-integral.

\begin{theorem}[Mean value theorem for the diamond integral]
Let $f,g:\mathbb{T}\rightarrow \mathbb{R}$ be bounded and $\diamondsuit$-integrable
functions on $\left[ a,b\right]_{\mathbb{T}}$, and let $g$ be nonnegative or nonpositive on
$\left[ a,b\right]_{\mathbb{T}}$. Let $m$ and $M$ be the infimum and supremum,
respectively, of function $f$. Then, there exists
a real number $K$ satisfying the inequalities
$m\leqslant K\leqslant M$ such that
\begin{equation*}
\int_{a}^{b}\left( fg\right) \left( t\right) \diamondsuit t
=K\int_{a}^{b}g\left( t\right) \diamondsuit t.
\end{equation*}
\end{theorem}

\begin{proof}
Without loss of generality, we suppose that $g$ is nonnegative
on $\left[ a,b\right]_{\mathbb{T}}$.
Since, for all $t\in \left[ a,b\right]_{\mathbb{T}}$,
$m\leqslant f\left( t\right) \leqslant M$
and $g\left( t\right) \geqslant 0$, then
$m g\left( t\right) \leqslant f\left( t\right) g\left( t\right)
\leqslant M g\left( t\right)$
for all $t\in \left[ a,b\right]_{\mathbb{T}}$.
Each of the functions $mg$, $fg$ and $Mg$ is
$\diamondsuit$-integrable from $a$ to $b$ and,
by Theorem~\ref{ts:propriedades:integral}, one has
\begin{equation*}
m\int_{a}^{b}g\left( t\right) \diamondsuit t\leqslant \int_{a}^{b}f\left(
t\right) g\left( t\right) \diamondsuit t\leqslant M\int_{a}^{b}g\left(
t\right) \diamondsuit t.
\end{equation*}
If $\displaystyle \int_{a}^{b}g\left( t\right) \diamondsuit t=0$, then
$\displaystyle \int_{a}^{b}f\left( t\right) g\left( t\right)
\diamondsuit t=0$ and we can choose any $K\in [m, M]$.
If $\displaystyle \int_{a}^{b}g\left( t\right) \diamondsuit t>0$, then
$\displaystyle m\leqslant \frac{\int_{a}^{b}f\left( t\right) g\left( t\right) \diamondsuit t}{
\int_{a}^{b}g\left( t\right) \diamondsuit t}\leqslant M$,
and we choose
$K:=\displaystyle \frac{\int_{a}^{b}f\left( t\right) g\left( t\right)
\diamondsuit t}{\int_{a}^{b}g\left( t\right) \diamondsuit t}$.
\end{proof}

We now present $\diamondsuit $-versions
of H\"{o}lder's, Cauchy-Schwarz's
and Minkowski's inequalities.

\begin{theorem}[H\"{o}lder's inequality for the diamond integral]
\label{ts:Holders:Inequality:TS}
If $f,g:\mathbb{T}\rightarrow \mathbb{R}$ are
$\diamondsuit$-integrable on $\left[ a,b\right]_{\mathbb{T}}$, then
\begin{equation*}
\int_{a}^{b}\left\vert f\left( t\right) g\left( t\right) \right \vert
\diamondsuit t\leqslant \left( \int_{a}^{b}\left \vert f\left( t\right)
\right \vert ^{p}\diamondsuit t\right) ^{\frac{1}{p}}\left(
\int_{a}^{b}\left \vert g\left( t\right)
\right\vert^{q}\diamondsuit t\right)^{\frac{1}{q}},
\end{equation*}
where $p>1$ and $q=\displaystyle \frac{p}{p-1}$.
\end{theorem}

\begin{proof}
For $\lambda,\beta\in \mathbb{R}_{0}^{+}$ and
$p,q$ such that $p>1$ and $\frac{1}{p}+\frac{1}{q}=1$, the
following inequality holds (Young's inequality):
\begin{equation*}
\lambda ^{\frac{1}{p}}\beta ^{\frac{1}{q}}\leqslant \frac{\lambda }{p}
+ \frac{\beta}{q}.
\end{equation*}
Without loss of generality, let us suppose that
$\left( \int_{a}^{b}\left \vert f\left( t\right) \right \vert^{p}\diamondsuit
t\right) \left( \int_{a}^{b}\left \vert g\left(
t\right) \right \vert^{q}\diamondsuit t\right) \neq 0$
(note that both integrals exist by Theorem~\ref{ts:propriedades:integral}). Define
\begin{equation*}
\lambda \left( t\right) :=\frac{\left \vert f\left( t\right)
\right \vert ^{p}}{\int_{a}^{b}\left \vert
f\left( \tau \right) \right \vert ^{p}\diamondsuit \tau}
\text{ \ and \ }\beta \left( t\right) :=\frac{\left \vert g\left( t\right)
\right \vert ^{q}}{\int_{a}^{b}\left \vert g\left( \tau \right) \right \vert
^{q}\diamondsuit \tau}.
\end{equation*}
Since both functions $\lambda $ and $\beta$ are $\diamondsuit$-integrable
on $\left[ a,b\right]_\mathbb{T} $, then
\begin{equation*}
\begin{split}
\int_{a}^{b} & \frac{\left \vert f\left( t\right) \right \vert }{\left(
\int_{a}^{b}\left \vert f\left( \tau \right) \right \vert^{p}\diamondsuit \tau
\right)^{\frac{1}{p}}} \frac{\left \vert g\left( t\right)
\right\vert}{\left( \int_{a}^{b}\left \vert g\left( \tau \right)
\right \vert ^{q}\diamondsuit\tau \right) ^{\frac{1}{q}}}\diamondsuit t\\
& \qquad =\int_{a}^{b}\left(\lambda \left( t\right)\right)^{\frac{1}{p}}
\left(\beta \left( t\right)\right)^{\frac{1}{q}}\diamondsuit t \\
& \qquad \leqslant \int_{a}^{b}\left( \frac{\lambda \left( t\right) }{p}
+\frac{\beta \left( t\right)}{q}\right) \diamondsuit t \\
& \qquad = \frac{1}{p}\int_{a}^{b}\left( \frac{\left \vert f\left( t\right)
\right \vert ^{p}}{\int_{a}^{b}\left \vert f\left( \tau \right) \right \vert
^{p}\diamondsuit \tau}\right) \diamondsuit t+\frac{1}{q}\int_{a}^{b}\left(
\frac{\left \vert g\left( t\right) \right \vert ^{q}}{\int_{a}^{b}\left \vert
g\left( \tau \right) \right \vert ^{q}\diamondsuit \tau}\right) \diamondsuit t
=\frac{1}{p}+\frac{1}{q} =1,
\end{split}
\end{equation*}
proving the intended result.
\end{proof}

\begin{corollary}[Cauchy-Schwarz's inequality for the diamond integral]
If $f,g:\mathbb{T}\rightarrow \mathbb{R}$ are $\diamondsuit$-integrable
on $\left[ a,b\right]_{\mathbb{T}}$, then
\begin{equation*}
\int_{a}^{b}\left \vert f\left( t\right) g\left( t\right) \right \vert
\diamondsuit t\leqslant \sqrt{\left( \int_{a}^{b}\left \vert f\left( t\right)
\right \vert ^{2}\diamondsuit t\right) \left( \int_{a}^{b}\left \vert g\left(
t\right) \right \vert ^{2}\diamondsuit t\right) }\text{.}
\end{equation*}
\end{corollary}

\begin{proof}
This is a particular case of Theorem~\ref{ts:Holders:Inequality:TS} where $p=2=q$.
\end{proof}

\begin{theorem}[Minkowski's inequality for the diamond integral]
If $f,g:\mathbb{T}\rightarrow \mathbb{R}$ is
$\diamondsuit$-integrable on $\left[ a,b\right]_{\mathbb{T}}$ and $p>1$, then
\begin{equation*}
\left( \int_{a}^{b}\left \vert f\left( t\right) +g\left( t\right) \right \vert^{p}
\diamondsuit t\right) ^{\frac{1}{p}}\leqslant \left(
\int_{a}^{b}\left \vert f\left( t\right) \right \vert ^{p}\diamondsuit
t\right) ^{\frac{1}{p}}+\left( \int_{a}^{b}\left \vert g\left( t\right)
\right \vert ^{p}\diamondsuit t\right) ^{\frac{1}{p}}.
\end{equation*}
\end{theorem}

\begin{proof}
If $\displaystyle \int_{a}^{b}\left \vert f\left( t\right)
+g\left( t\right) \right \vert^{p}\diamondsuit t=0$, then the result is trivial.
Suppose $\displaystyle \int_{a}^{b}\left \vert f\left( t\right)
+g\left( t\right) \right \vert^{p}\diamondsuit t\neq0$. Since
\begin{equation*}
\begin{split}
\int_{a}^{b}\left \vert f\left( t\right) +g\left( t\right) \right \vert^{p}
\diamondsuit t&=\int_{a}^{b}\left \vert f\left( t\right) +g\left( t\right)
\right \vert ^{p-1}\left \vert f\left( t\right) +g\left( t\right) \right \vert
\diamondsuit t \\
&\leqslant \int_{a}^{b}\left \vert f\left( t\right) \right \vert \left \vert
f\left( t\right) +g\left( t\right) \right \vert ^{p-1}\diamondsuit
t+\int_{a}^{b}\left \vert g\left( t\right) \right \vert \left \vert f\left(
t\right) +g\left( t\right) \right \vert ^{p-1}\diamondsuit t,
\end{split}
\end{equation*}
then, by H\"{o}lder's inequality (Theorem~\ref{ts:Holders:Inequality:TS}) with
$q=\displaystyle \frac{p}{p-1}$, we obtain that
\begin{equation*}
\int_{a}^{b}\left \vert f\left( t\right) +g\left( t\right) \right \vert^{p}\diamondsuit t
\leqslant \left[ \left( \int_{a}^{b}\left \vert f\left( t\right) \right \vert
^{p}\diamondsuit t\right) ^{\frac{1}{p}}+\left( \int_{a}^{b}\left \vert
g\left( t\right) \right \vert ^{p}\diamondsuit t\right) ^{\frac{1}{p}}\right]
 \left( \int_{a}^{b}\left \vert f\left( t\right) +g\left( t\right)
\right \vert ^{\left( p-1\right) q}\diamondsuit t\right) ^{\frac{1}{q}}.
\end{equation*}
Dividing both sides by
$\left(
\int_{a}^{b}\left \vert f\left( t\right) +g\left( t\right) \right \vert^{p}
\diamondsuit t\right) ^{\frac{1}{q}}$,
we arrive to the intended inequality:
\begin{equation*}
\left( \int_{a}^{b}\left \vert f\left( t\right) +g\left( t\right) \right \vert^{p}
\diamondsuit t\right) ^{\frac{1}{p}}\leqslant \left(
\int_{a}^{b}\left \vert f\left( t\right) \right \vert ^{p}\diamondsuit
t\right) ^{\frac{1}{p}}+\left( \int_{a}^{b}\left \vert g\left( t\right)
\right \vert ^{p}\diamondsuit t\right) ^{\frac{1}{p}}.
\end{equation*}
\end{proof}


\section{Conclusion}
\label{sec:coc}

Combined delta and nabla derivatives, as well as combined integrals,
are increasingly getting more attention in approximating functions
and solutions of differential equations. The recent symmetric derivative on time scales
\cite{Brito_da_Cruz_3} unifies the symmetric derivatives of classical analysis and quantum calculus.
Moreover, it is a generalization of the delta and nabla derivatives on time scales.
It is important to note that the symmetric derivative on time scales is different
from the delta and nabla derivatives: for example, the absolute value function is neither delta
nor nabla differentiable but it is symmetric differentiable.
When a function is simultaneously delta and nabla differentiable, the symmetric derivative
can then be written as a time dependent combination of those derivatives.
In this paper we introduce a new type of integral on time scales,
the diamond integral, based on the definition of the symmetric derivative \cite{Brito_da_Cruz_3}.
Instead of a constant parameter, like in the definition of the diamond-$\alpha$ integral
\cite{Sheng,Sheng_2}, we propose a new version of the diamond integral
with a time dependent parameter. The diamond integral is the Riemman integral
when the time scale is $\mathbb{T=R}$ and is an arithmetic average
of the delta and nabla integrals when the time scale is $\mathbb{T=Z}$
or any other time scale with the forward graininess function $\mu(t) :=\sigma (t)- t$
or the backward graininess function $\nu(t) :=t- \rho(t)$ constant.
We prove that this new integral satisfies important properties such
as the mean value theorem and H\"{o}lder and Minkowski type inequalities.
It is well known that integral inequalities on time scales play a major role
in the development of other areas of mathematics. For example,
in the calculus of variations, Holder and Minkowski type inequalities
can be used to find explicitly the extremizers for some classes
of variational problems on time scales \cite{BFT10}.
We trust that our diamond integral on time scales
is interesting and useful, and will lead to subsequent investigations
with important applications.


\section*{Acknowledgments}

This work was supported by {\it FEDER} funds through
{\it COMPETE} --- Operational Programme Factors of Competitiveness
(``Programa Operacional Factores de Competitividade'')
and by Portuguese funds through the {\it Center for Research
and Development in Mathematics and Applications} (University of Aveiro)
and the Portuguese Foundation for Science and Technology
(``FCT --- Funda\c{c}\~{a}o para a Ci\^{e}ncia e a Tecnologia''),
within project PEst-C/MAT/UI4106/2011 with COMPETE number FCOMP-01-0124-FEDER-022690.
Brito da Cruz was also supported by FCT through the Ph.D. fellowship
SFRH/BD/33634/2009. The authors are grateful to two anonymous referees 
for valuable comments and helpful suggestions.




\begin{thebibliography}{99}

\bibitem{Hilger}
B. Aulbach\ and\ S. Hilger,
A unified approach to continuous and discrete dynamics,
in {\it Qualitative theory of differential equations (Szeged, 1988)}, 37--56,
Colloq. Math. Soc. J\'anos Bolyai, 53 North-Holland, Amsterdam, 1990.

\bibitem{Hilger:90}
S. Hilger,
Analysis on measure chains---a unified approach to continuous and discrete calculus,
Results Math. {\bf 18} (1990), no.~1-2, 18--56.

\bibitem{Bohner_1}
M. Bohner\ and\ A. Peterson,
{\it Dynamic equations on time scales},
Birkh\"auser Boston, Boston, MA, 2001.

\bibitem{Bohner_2}
M. Bohner\ and\ A. Peterson,
{\it Advances in dynamic equations on time scales},
Birkh\"auser Boston, Boston, MA, 2003.

\bibitem{Sheng}
Q. Sheng, M. Fadag, J. Henderson\ and\ J. M. Davis,
An exploration of combined dynamic derivatives on time scales and their applications,
Nonlinear Anal. Real World Appl. {\bf 7} (2006), no.~3, 395--413.

\bibitem{Brito_da_Cruz_3}
A. M. C. Brito da Cruz, N. Martins\ and\ D. F. M. Torres,
Symmetric differentiation on time scales,
Appl. Math. Lett. {\bf 26} (2013), no.~2, 264--269.
{\tt arXiv:1209.2094}

\bibitem{Thomson}
B. S. Thomson,
{\it Symmetric properties of real functions},
Monographs and Textbooks in Pure and Applied Mathematics,
183, Dekker, New York, 1994.

\bibitem{Sheng_2}
J. W. Rogers, Jr.\ and\ Q. Sheng,
Notes on the diamond-$\alpha$ dynamic derivative on time scales,
J. Math. Anal. Appl. {\bf 326} (2007), no.~1, 228--241.

\bibitem{Sheng_3}
Q. Sheng,
Hybrid approximations via second order combined dynamic derivatives on time scales,
Electron. J. Qual. Theory Differ. Equ. {\bf 2007} (2007), no.~17, 13~pp.

\bibitem{Malinowska}
A. B. Malinowska\ and\ D. F. M. Torres,
On the diamond-alpha Riemann integral and mean value theorems on time scales,
Dynam. Systems Appl. {\bf 18} (2009), no.~3-4, 469--481.
{\tt arXiv:0804.4420}

\bibitem{Rui}
M. R. Sidi Ammi, R. A. C. Ferreira\ and\ D. F. M. Torres,
Diamond-$\alpha$ Jensen's inequality on time scales,
J. Inequal. Appl. {\bf 2008} (2008), Art. ID 576876, 13~pp.
{\tt arXiv:0712.1680}

\bibitem{Moz:Tor:09}
D. Mozyrska\ and\ D. F. M. Torres,
Diamond-alpha polynomial series on time scales,
Int. J. Math. Stat. {\bf 5} (2009), A09, 92--101.
{\tt arXiv:0805.0274}

\bibitem{MR2609350}
M. R. Sidi Ammi\ and\ D. F. M. Torres,
H\"older's and Hardy's two dimensional diamond-alpha inequalities on time scales,
An. Univ. Craiova Ser. Mat. Inform. {\bf 37} (2010), no.~1, 1--11.
{\tt arXiv:0910.4115}

\bibitem{Thomson_2}
G. E. Cross\ and\ B. S. Thomson,
Symmetric integrals and trigonometric series,
Dissertationes Math. (Rozprawy Mat.) {\bf 319} (1992), 49~pp.

\bibitem{Kac}
V. Kac\ and\ P. Cheung,
{\it Quantum calculus},
Universitext, Springer, New York, 2002.

\bibitem{Brito_da_Cruz_2}
A. M. C. Brito da Cruz\ and\ N. Martins,
The $q$-symmetric variational calculus,
Comput. Math. Appl. {\bf 64} (2012), no.~7, 2241--2250.

\bibitem{Brito_da_Cruz_4}
A. M. C. Brito da Cruz, N. Martins\ and\ D. F. M. Torres,
Hahn's symmetric quantum variational calculus,
Numer. Algebra Control Optim. {\bf 3} (2013), no.~1, 77--94.
{\tt arXiv:1209.1530}

\bibitem{BFT10}
M. J. Bohner, R. A. C. Ferreira\ and\ D. F. M. Torres,
Integral inequalities and their applications to the calculus of variations on time scales,
Math. Inequal. Appl. {\bf 13} (2010), no.~3, 511--522.
{\tt arXiv:1001.3762}

\end{thebibliography}
\end{document}